\documentclass[11pt]{article}
\usepackage{amsmath, amsfonts, amssymb, amsthm}
\newtheorem{Theorem}{Theorem}[section]

\newtheorem{Lemma}[Theorem]{Lemma}

\newtheorem{Open Problem}[Theorem]{Open Problem}

\newcommand{\mysection}[1]{\section{#1}\setcounter{equation}{0}}

\begin{document}

\title{ Existence of positive ground state solutions for the coupled Choquard system with potential}

\author{Jianqing Chen$^{1,2}$\ and Qian Zhang$^{1,}$\thanks{Corresponding 
author:\ qzhang\_fjnu@163.com} \\
\small  1.\  School of Mathematics and Statistics,  Fujian Normal University, \\
\small  Qishan Campus, Fuzhou 350117, P. R. China\\
\small  2.\  FJKLMAA and Center for Applied Mathematics of Fujian Province(FJNU), \\
\small  Fuzhou, 350117, P. R. China   }

\date{}

\maketitle
\noindent {\bf Abstract}: In this paper, we study the following coupled Choquard system in $\mathbb R^N$:
$$\left\{\aligned&-\Delta u+A(x)u=\frac{2p}{p+q} \bigl(I_\alpha\ast |v|^q\bigr)|u|^{p-2}u,\\
&-\Delta v+B(x)v=\frac{2q}{p+q}\bigl(I_\alpha\ast|u|^p\bigr)|v|^{q-2}v,\\
&\ u(x)\to0\ \ \hbox{and}\ \ v(x)\to0\ \ \hbox{as}\ |x|\to\infty,\endaligned\right.$$
where $\alpha\in(0,N)$ and $\frac{N+\alpha}{N}<p,\ q<2_*^\alpha$, in which  $2_*^\alpha$ denotes $\frac{N+\alpha}{N-2}$ if $N\geq 3$ and  $2_*^\alpha := \infty$ if $N=1,\ 2$.  The function $I_\alpha$ is a Riesz potential. By using Nehari manifold method, we obtain the existence of positive ground state solution in the case of bounded potential and periodic potential respectively. In particular, the nonlinear term includes the well-studied case $p=q$ and $u(x)=v(x)$, and the less-studied case $p\neq q$ and $u(x)\neq v(x)$. Moreover it seems to be the first existence result  for the  case of $p\neq q$.
\medskip


\medskip


\mysection {Introduction}
This paper is dedicated to the following Choquard system
\begin{equation}\label{eq1.1}
\left\{\aligned &-\Delta u + A(x)u  =\frac{2p}{p+q} \bigl(I_\alpha\ast |v|^q\bigr) |u|^{p-2}u,\\
&-\Delta v+B(x)v= \frac{2q}{p+q}\bigl(I_\alpha \ast |u|^p\bigr) |v|^{q-2}v,\\
&\ u(x)\to 0\ \ \hbox{and}\ \  v(x)\to 0\ \ \hbox{as}\ |x|\to \infty,\endaligned\right.
\end{equation}
where $u:=u(x),v:=v(x)$ are real valued functions on $\mathbb R^{N}(N\geq1)$. $I_{\alpha}:\mathbb R^{N}\backslash\{0\}\rightarrow\mathbb R$ is the Riesz potential defined by
$$I_{\alpha}(x)=\frac{A_{\alpha}}{|x|^{N-\alpha}}\ \ \hbox{with}\ \ A_{\alpha}=\frac{\Gamma(\frac{N-\alpha}{2})}{2^{\alpha}\pi^{\frac{N}{2}\Gamma(\frac{\alpha}{2})}},\ \ \alpha\in(0,N),$$
where $\Gamma$ denotes the classical Gamma function and $*$ the convolution on the Euclidean space $\mathbb R^{N}$. When $p=q$ and $u(x)=v(x)$, the system (\ref{eq1.1}) is related to the following single Choquard equation:
\begin{equation}\label{eq1.2}
-\Delta u+A(x)u=(I_{\alpha}\ast|u|^p)|u|^{p-2}u \ \ \hbox{in}\ \mathbb R^N.
\end{equation}
When $A(x)\equiv1$, Choquard equation (\ref{eq1.2}) is firstly appeared  in a work by Pekar describing the quantum mechanics of a polaron at rest \cite{pekar}. In the case of $N=3,\alpha=2$ and $p=2$, Choquard describe an electron trapped in its own hole, in a certain approximation to Hartree-Fock theory of one component plasma \cite{lieb}. The existence and uniqueness of the ground state solution was first considered by Lieb in 1976 \cite{lieb}. Later Lions \cite{li,lio} obtained the existence and multiplicity of solutions to (\ref{eq1.2}). In 1996, Penrose proposed a model of self-gravitating matter, in a programme in which quantum state reduction is understood as a gravitational phenomenon \cite{moroz}. Since the Choquard equation (\ref{eq1.2}) has been widely investigated by employing variational methods. For more interesting existence and qualitative properties of solutions to \eqref{eq1.2}, we refer the reader to \cite{morozs,ms,mos} and references therein.

In recent years, there has been increasing attention to equations like \eqref{eq1.2} on the existence of positive solutions and ground state solutions. When $A(x)\equiv1$, by using the concentration-compactness and Br\'{e}zis-Lieb lemmas, Moroz and Schaftingen \cite{mor} showed the existence of ground state solutions to \eqref{eq1.2} for $\frac{N+\alpha}{N}<p<\frac{N+\alpha}{N-2}$. Here and after, the  $2_*^\alpha$ denotes $\frac{N+\alpha}{N-2}$ if $N\geq 3$ and  $2_*^\alpha := \infty$ if $N=1,\ 2$.  We would also like to mention the papers \cite{ac,ma,mors} for related topics. In \cite{88}, Fr\"{o}hlich and  Lenzmann studied the existence of a positive radially symmetric ground state solution for \eqref{eq1.2} with a  confining potential:
\begin{equation}\label{eq1.03}
\lim_{|x|\rightarrow\infty}A(x)=+\infty.
\end{equation}
Such kind of hypotheses was introduced to guarantee the compactness of embedding  of
$E:=\left\{u\in H^{1}(\mathbb R^{N})\  | \ \int_{\mathbb R^{N}}A(x)u^{2}<\infty\right\}\hookrightarrow L^{s}(\mathbb R^{N})$
for $ 2\leq s< \frac{2N}{N-2}$.
The specific case of  $A(x)=|x|^2$ and $\frac{N+\alpha}{N}<p< 2_*^\alpha$, Feng \cite{80} has been studied the existence of a ground state solution. Under the assumption \eqref{eq1.03}, Schaftingen and Xia \cite{schx} obtained the sign-changing solution for the equation \eqref{eq1.2} in the case of $p\in[2, 2_*^\alpha)$. Clapp and Salazar \cite{c} gave the existence of positive solution for the nonlinear Choquard equation
$$-\Delta u+A(x)u=(|x|^{-\alpha}\ast|u|^p)|u|^{p-2}u,\  \ u \in H^1_0 (\Omega),$$
where $\Omega$ is an exterior domain of $\mathbb R^N$ with some symmetries, $p\in[2,\frac{2N-\alpha}{N-2})$, $A(x)$ is continuous, radially symmetric and
$$\inf_{x\in\mathbb R^ N}A(x) >0,\ \ A(x)\rightarrow A_\infty > 0 \ \ \hbox{as}\ \ |x|\rightarrow+\infty.$$

When $A(x)\equiv1$, a perturbed version of (\ref{eq1.2}) with the following form
\begin{equation}\label{eq1.3}
-\Delta u + u = (I_\alpha\ast|u|^p )|u|^{p-2} u + |u|^{ q-2} u,\ \  x\in\mathbb  R^N
\end{equation}
has been studied in \cite{chenguo}, where the  existence of solutions is obtained for $N=3$, $0<\alpha<1$, $p=2$ and $4\leq q<6$. Vaira \cite{vai1} has obtained a positive ground state solution when $N=3,\ \alpha=2,\ p=2,\ q\in(2,6)$, and Vaira \cite{vai2} further studied the nondegeneracy of the ground state solution in a special case $q=3$. Via the mountain pass theorem and Poho\v{z}aev identity, Li, Ma and Zhang in \cite{lm} studied  \eqref{eq1.3} for $\frac{N+\alpha}{N}<p<2_*^\alpha,\ 2<q<\frac{2N}{N-2}$, and proved the existence of ground state solution of mountain pass type. For more results of \eqref{eq1.3}, we refer to \cite{00,01,02} and the references therein. Recently, many researchers are interested in the elliptic system of Choquard type. When $\frac{N+\alpha}{N}<p=q<2_*^\alpha$, Chen and Liu \cite{chenpeng} have obtained the existence of positive ground state solutions to the following linearly coupled system:
\begin{equation}\label{eq1.4}
\left\{\aligned&-\Delta u+u=(I_\alpha\ast|u|^p)|u|^{p-2}u+\lambda v,\ \ x\in\mathbb R^N,\\
&-\Delta v+v=(I_\alpha\ast|v|^q)|v|^{q-2}v+\lambda u,\ \ x\in\mathbb R^N.\endaligned\right.
\end{equation}
Yang et al. \cite{yang} have proven the existence of positive ground state solutions when $p=q$ reaches the critical exponent. A slightly general version of (\ref{eq1.4}) was studied by Xu et al. \cite{xu}, where the authors have proven that the following system:
$$\left\{\aligned&-\Delta u+\lambda_1u=(I_\alpha\ast F(u))F'(u)+\beta v,\\
&-\Delta v+\lambda_2v=(I_\alpha\ast F(v))F'(v)+\beta u,\endaligned\right.$$
admits a nontrivial vector solution with the help of some conditions on $F'$ and Schwarz symmetrization method. Albuquerque, Silvab and Sousa \cite{al} considered the existence of positive solutions for the fractional coupled Choquard-type system. In \cite{al,chenpeng,xu,yang}, all results are positive solutions or ground state solutions of the linear coupled type Choquard system. Sun and Chang \cite{sun} obtained ground state solution of Choquard type with general nonlinearity.

But we do not see any results to (\ref{eq1.1}) in the case of $p\neq q$ and $A(x)\neq1$. In our recent work \cite{chenzhang}, we consider the ground state solutions of Choquard system with strongly indefinite structure. In this paper, we continue to study the ground state solutions of the Choquard system, and give the existence of the ground state solutions under the bounded potential and periodic potential.

The Choquard system \eqref{eq1.1} is variational in nature. A starting point is the following symmetric property
$$\aligned\int_{\mathbb R^N}\bigl(I_\alpha\ast|v|^q\bigr)|u|^p&=\int_{\mathbb R^N}\int_{\mathbb R^N}\frac{A_\alpha|v(x)|^q|u(y)|^p}{|x-y|^{N-\alpha}}dxdy\\
&=\int_{\mathbb R^N}\bigl(I_\alpha\ast|u|^p\bigr)|v|^q,\ \ u,v\in H^1(\mathbb R^N).\endaligned$$
Therefore the action functional $\mathcal{I}$ associated to the Choquard system (\ref{eq1.1}) is defined for each function $(u,v)$ in the space $H:=H^{1}(\mathbb R^N)\times H^{1}(\mathbb R^N)$ by
$$\aligned\mathcal{I}(u,v)&=\frac{1}{2}\int_{\mathbb R^N}(|\nabla u|^2+A(x)|u|^2+|\nabla v|^2+B(x)|v|^2)\\
&\ \ \ \ -\frac{2}{p+q}\int_{\mathbb R^N}(I_\alpha\ast|u|^p)|v|^q.\endaligned$$
In view of the Hardy-Littlewood-Sobolev inequality  \cite[theorem 4.3]{lil}, which states that if $s\in(1,\frac{N}{\alpha})$ then for every $v \in L^s(\mathbb R^N)$, $I_\alpha\ast v\in L^\frac{Ns}{N- \alpha s}(\mathbb R^N)$ and
\begin{equation}\label{eqHLS}
\int_{\mathbb R^N}|I_\alpha \ast v|^\frac{N s}{N-\alpha s}\le C\Bigl(\int_{\mathbb R^N}|v|^s \Bigr)^\frac{N}{N-\alpha s},
\end{equation}
and of the classical Sobolev embedding, the action functional $\mathcal{I}$ is well-defined and continuously differentiable whenever
$\frac{N+\alpha}{N}<p,\ q <2_*^\alpha.$ A natural constraint for the equation is the Nehari constraint
$\mathcal{P}(u,v):=\langle\mathcal{I}'(u,v),(u,v)\rangle=0,$
which leads to search for solutions to \eqref{eq1.1} by minimizing the action functional on the Nehari manifold
$$\mathcal{N}_0=\bigl\{(u,v)\in H\backslash\{(0,0)\}\ |\ \mathcal{P}(u,v)=0\bigr\}.$$
A nontrivial solution $(u,v)\in H $ of \eqref{eq1.1} is called a ground state solution if
\begin{equation}\label{eqc}
\aligned\mathcal{I}(u,v)&=c_0:=\inf_{(u,v)\in\mathcal{N}_0}\mathcal{I}(u,v).\endaligned
\end{equation}

Note that if $(u,0)$ is a solution to \eqref{eq1.1}, then $u=0$; and if  $(0,v)$ is a solution to \eqref{eq1.1}, then $v=0$. We have that
$$c_0=\inf \{\mathcal{I}(u,v)\ | \ u\neq0,\ v\neq 0, \ \mathcal{P}(u,v) =0\}. $$

Before state our main results, we make the following assumptions

$(F_{0})$ $A\in C(\mathbb R^{N},\mathbb R^{+}),$ $A(x)\geq A_{0}>0;$

$(F_{1})$ $ A(x)\leq A_{\infty}=\lim\limits_{|x|\rightarrow\infty}A(x)<+\infty$;

$(F_{2})$ $A(x)$ is $\tau_{i}$-periodic in $x_{i}$, $\tau_{i}>0$, $1\leq i\leq N$;

$(F_{3})$ $B\in C(\mathbb R^{N},\mathbb R^{+}),$ $B(x)$ is $\tau_{i}$-periodic in $x_{i}$, $\tau_{i}>0$, $1\leq i\leq N,$ $B(x)\geq B >0$.

Our  results are that this level $c_0$ is achieved.
\begin{Theorem}\label{th1.1}
Assume that $(F_{0})$, $(F_{1})$ and $B(x)\equiv B>0 .$ If $\frac{N+\alpha}{N}<p,\ q< 2_*^\alpha$ and $p\neq q$, then there exists a positive ground state solution $(u,v)\in H$ to the system \eqref{eq1.1} such that $\mathcal{I}(u,v)=c_0$.
\end{Theorem}

\begin{Theorem}\label{th1.2}
Assume that $(F_{0})$, $(F_{2})$ and $(F_{3})$. If $\frac{N+\alpha}{N}<p,\ q< 2_*^\alpha$ and $p\neq q$, then there exists a positive ground state solution $(u,v)\in H$ to the \eqref{eq1.1} such that $\mathcal{I}(u,v)=c_0$.
\end{Theorem}

\noindent {\bf Corollary 1.3.} If $A(x)$ and $B(x)$ are positive constants, one can still obtain the same results as  Theorem \ref{th1.1} and Theorem  \ref{th1.2} for system (\ref{eq1.1}).

\vskip4pt
\noindent {\bf Remark 1.4.}  Compared with related results, there are some differences and difficulties in our proofs: (1) By using the energy of any sign changing solution of the problem is strictly less than twice the least energy to ensure the compactness of Palais-Smale sequence, the authors in \cite{10} proved the existence of ground state solution. Here, we use the Nehari manifold method to prove it in Section 2. Moreover, our results include the case of bounded potential and periodic potential. (2) Different from \cite{guiguo,13,wangguo}, in this paper, we have to overcome the difficulty caused by the nonlinear coupled Choquard term $\int_{\mathbb R^N}(I_\alpha\ast|u|^p)|v|^q$. (3) In \cite{chenpeng,xu,yang}, they all obtain the positive solutions or ground state solutions of the Choquard system with $p=q$ and $u(x)=v(x)$. However, the positive ground state solution of the Choquard system (\ref{eq1.1}) with $p\neq q$ and $u(x)\neq v(x)$.
\vskip4pt

The rest of the paper is organized as follows. In Section 2, we set the variational framework for system \eqref{eq1.1} and some preliminary results. Section 3 is devoted
to studying the existence of ground state solutions to system \eqref{eq1.1} under the bounded potential. The existence of ground state solutions to system \eqref{eq1.1} under the periodic potential is proved in Section 4.

\mysection {Preliminaries and the proof of Theorem \ref{th1.1}}

Throughout this paper, $\|u\|_{H^1}$ and $|u|_r$ denote the usual norm of $H^1(\mathbb R^N)$ and $L^r(\mathbb R^N)$ for $r>1$, respectively. Let
$$\|(u,v)\|^2:=\int_{\mathbb R^{N}}(|\nabla u|^{2}+|\nabla v|^{2} + A(x)|u|^{2}+B(x)| v|^{2}).$$
This norm is equivalent to the norm of
$$\|(u,v)\|_H^2:=\int_{\mathbb R^{N}}(|\nabla u|^{2}+|\nabla v|^{2} + |u|^{2}+ |v|^{2}).$$
For convenience, $C$ and $C_{i}(i=1,2,\ldots)$ denote (possibly different) positive constants, $\int_{\mathbb R^{N}}g$ denotes the integral $\int_{\mathbb R^{N}}g(z)dz$. The $\rightarrow$ denotes strong convergence. The $\rightharpoonup$ denotes weak convergence.

Let $t\in\mathbb R^{+}$ and $(u,v)\in H$, we have
$$\aligned\mathcal{I}(tu,tv)&=\frac{1}{2}t^{2}\int_{\mathbb R^N}(|\nabla u|^2+|\nabla v|^2+A(x)|u|^2+B(x)|v|^2)\\
&\ \ \ \  -\frac{2 }{p+q}t^{p+q}\int_{\mathbb R^N}(I_\alpha\ast|u|^p)|v|^q.\endaligned$$
Let $h_{uv}(t):=\mathcal{I}(tu,tv)$. Since $p+q>\frac{2(N+\alpha)}{N}>2,$ we see that $h(t)>0$ for $t>0$ small enough and $h(t)\rightarrow-\infty$ as $t\rightarrow+\infty,$ which implies that $h(t)$ attains its maximum.
\begin{Lemma}\label{le24}
Let $\theta_1,\ \theta_2>0$. For every $t\geq0,$ define $h_{uv}(t):=\theta_{1}t^{2}-\theta_{2}t^{p+q}$. Then $h_{uv}$ has a unique critical point which corresponds to its maximum.
\end{Lemma}
\begin{proof}
Since $p+q>\frac{2(N+\alpha)}{N}>2,$ it is easy to check that $h_{uv}$ has a maximum. Through a simple calculation, we get
$$h_{uv}'(t)=2\theta_{1}t-\theta_{2}(p+q)t^{p+q-1}.$$
Since $h_{uv}'(t)\rightarrow-\infty$ as $t\rightarrow+\infty$ and is positive for $t>0$ small, we obtain that there is $t>0$ such that $h_{uv}'(t)=0.$ The uniqueness of the critical point of $h_{uv}$ follows from the fact that the equation
$$h_{uv}'(t)=2\theta_{1}t-(p+q)\theta_{2}t^{p+q-1}=0$$
has a unique positive solution $\bigl(\frac{2\theta_{1}}{(p+q)\theta_{2}}\bigr)^{\frac{1}{p+q-2}}$.
\end{proof}

\begin{Lemma}\label{le22} For any $(u,v)\in H\backslash\{(0,0)\}$, there exists a unique $\bar{t}:=t(u,v)>0$ such that $h$ attains its maximum at $\bar{t}$ and $c_0=\inf\limits_{(u,v)\in H}\max\limits_{t>0}\mathcal{I}(tu,tv).$ Moreover, if $\mathcal{P}(u,v)<0$, we get $\bar{t}\in(0,1)$.
\end{Lemma}
\begin{proof}
Let $(u,v)\in H\backslash\{(0,0)\}$ be fixed and for $t>0,$ by making the change of variable $s = t^{p+q}$, we get
$$\aligned h_{uv}(s) &=\frac{s^{\frac{2}{p+q}}}{2}\int_{\mathbb R^N}(|\nabla u|^2+|\nabla v|^2+A(x)|u|^2+B(x)|v|^2)\\
&\ \ \ \ -\frac{2s}{p+q}\int_{\mathbb R^N}(I_\alpha\ast|u|^p)|v|^q.\endaligned$$
Lemma \ref{le24} implies that $h_{uv}$ has a unique critical point $\bar{t}>0$ corresponding to its maximum, i.e., $h_{uv}(\bar{t})=\max\limits_{t>0}h_{uv}(t)$, $h_{uv}'(\bar{t})=0,$ and hence $\mathcal{P}(\bar{t}u,\bar{t}v)=0$ and $(\bar{t}u,\bar{t}v)\in \mathcal{N}_0.$ In addition, suppose that $\mathcal{P}(u,v)<0$. Then we have
$$\int_{\mathbb R^{N}}(|\nabla u|^{2}+A(x)|u|^2+|\nabla v|^{2}+B(x)|v|^{2})-2\int_{\mathbb R^{N}}(I_\alpha\ast|u|^{p})|v|^{q}<0, $$
and
$$\bar{t}^{2}\int_{\mathbb R^{N}}(|\nabla u|^{2}+A(x)|u|^2+|\nabla v|^{2}+B(x)|v|^{2})-2\bar{t}^{p+q}\int_{\mathbb R^{N}}(I_\alpha\ast|u|^{p})|v|^{q}=0.$$
Thus, by the above two formulas we can get that
$$\left(\bar{t}^{p+q}-\bar{t}^{2}\right)\int_{\mathbb R^{N}}(|\nabla u|^{2}+A(x)|u|^2+|\nabla v|^{2}+B(x)|v|^{2})<0,  $$
which implies that $\bar{t}<1$. This finishes the proof.
\end{proof}

\begin{Lemma}\label{le2.3}
The $\mathcal{N}_0$ is a $C^{1}$ manifold and every critical point of $\mathcal{I}|_{\mathcal{N}_0}$ is a critical point of $\mathcal{I}$ in $H.$
\end{Lemma}
\begin{proof} By Lemma \ref{le22}, $\mathcal{N}_0\neq\emptyset$.  The proof consists of four steps.

{\bf Step 1.}  $\mathcal{N}_0$ is bounded away from zero.

For any $(u,v)\in \mathcal{N}_0$, by using $\mathcal{P}(u,v)=0$, the semigroup property of the Riesz potential $I_\alpha=I_{\frac{\alpha}{2}}\ast I_{\frac{\alpha}{2}}$ \cite[Theorem 5.9 and Corollary 5.10]{lil}, Cauchy-Schwarz inequality, \eqref{eqHLS}, Sobolev and Young inequalities, we obtain that
\begin{equation}\label{009}
\aligned\|(u,v)\|^2&=\mathcal{P}(u,v)+2\int_{\mathbb R^N}(I_\alpha\ast|u|^p)|v|^q\\
&=2\int_{\mathbb R^{N}}(I_{\frac{\alpha}{2}}\ast|u|^p)(I_{\frac{\alpha}{2}}\ast|v|^q)\\
&\leq2\bigg(\int_{\mathbb R^{N}}(I_{\frac{\alpha}{2}}\ast|u|^p)^{2}\bigg)^\frac{1}{2}\bigg(\int_{\mathbb R^{N}}(I_{\frac{\alpha}{2}}\ast|v|^q)^{2}\bigg)^{\frac{1}{2}}\\
&=2\bigg(\int_{\mathbb R^{N}}(I_{\alpha}\ast|u|^p)|u|^p\bigg)^\frac{1}{2}\bigg(\int_{\mathbb R^{N}}(I_{\alpha}\ast|v|^q)|v|^q\bigg)^{\frac{1}{2}}\\
&\leq C_1\bigg(\int_{\mathbb R^{N}}|u|^{\frac{2Np}{N+\alpha}}\bigg)^\frac{N+\alpha}{2N}\bigg(\int_{\mathbb R^{N}}|v|^{\frac{2Nq}{N+\alpha}} \bigg)^\frac{N+\alpha}{2N}\\
&\leq C_2\|u\|_{H^1}^{p }\|v\|_{H^1}^q\\
&\leq C\|(u,v)\|^{2(p+q)}.\endaligned
\end{equation}
Hence, there is $C' >0$ such that $\|(u,v)\|\geq C'$. This proves that $\mathcal{N}_0$ is bounded away from
zero.

{\bf Step 2.}  $c_0>0$.

Set
$$\aligned \bar{\mathcal{I}}(u,v)&:=\frac{1}{2}\int_{\mathbb R^N}(|\nabla u|^2+A_0|u|^2+|\nabla v|^2+B|v|^2)\\
&\ \ \ \ -\frac{2}{p+q}\int_{\mathbb R^N}(I_\alpha\ast|u|^p)|v|^q,\endaligned$$
where $A_{0}$ and $B$ come from conditions $(F_{0}) $ and $(F_{3}) $, respectively. Obviously, $\bar{\mathcal{I}}(u,v)\leq \mathcal{I}(u,v)$, which implies that for any $u,\ v\neq0$,
$$\bar{c_0}:=\inf_{(u,v)\in H}\max_{t>0}\bar{\mathcal{I}}(tu ,tv )\leq \inf_{(u,v)\in H}\max_{t>0}\mathcal{I}(tu,tv)=c_0.$$
Then, it suffices to prove that $\bar{c_0}>0.$ Define
$$ \bar{\mathcal{N}_0}=\{(u,v)\in H\backslash\{(0,0)\} \ | \ g'_{uv}(1)=0\},$$
where $g_{uv}(t)=\bar{\mathcal{I}}(tu ,tv ).$ By Lemma \ref{le22} applied to $A(x)\equiv A_{0},\ B(x)\equiv B $, we get that
$$\bar{c_0}=\inf_{(u,v)\in\bar{\mathcal{N}_0}}\bar{\mathcal{I}}(u,v).$$
The {\bf Step 1} above shows that $\int_{\mathbb R^N}(|\nabla u |^2+|\nabla v |^2+A_0|u|^2+B|v |^2)$ is bounded away from zero on $\bar{\mathcal{N}_0} $. Then
$$\aligned \bar{\mathcal{I}}(u,v)&=\bar{\mathcal{I}}(u,v)-\frac{1}{p+q}\mathcal{P}(u,v)\\
&= \frac{p+q-2}{2(p+q)}\int_{\mathbb R^N}(|\nabla u |^2+|\nabla v |^2+A_{0}|u|^{2}+B|v |^2)\\
&> 0.\endaligned$$

{\bf Step 3.} $\mathcal{P}'(u,v)\neq0$ for all $(u,v)\in\mathcal{N}_0$, hence $\mathcal{N}_0$ is a $C^{1}$ manifold.

Just suppose that $\mathcal{P}'(u,v)=0$ for some $(u,v)\in\mathcal{N}_0$, then $(u,v)$ satisfies
$$\left\{\aligned &-\Delta u+A(x)u=p\bigl(I_\alpha\ast|v|^q\bigr)|u|^{p-2}u,\\
&-\Delta v+B(x)v=q\bigl(I_\alpha\ast|u|^p\bigr)|v|^{q-2}v,\endaligned\right.$$
$$\int_{\mathbb R^{N}}(|\nabla u|^{2}+A(x)| u|^{2})=p \int_{\mathbb R^{N}}(I_\alpha\ast|v|^q)|u|^{p},$$
and
$$\int_{\mathbb R^{N}}(|\nabla v|^{2}+B(x)| v|^{2})=q\int_{\mathbb R^{N}}(I_\alpha\ast|u|^{p})|v|^q.$$
Therefore,
$$\aligned 0&=\mathcal{P}(u,v)\\
&=\|(u,v)\|^2-2\int_{\mathbb R^{N}}(I_\alpha\ast|u|^{p}|v|^q)\\
&=(p+q-2)\int_{\mathbb R^{N}}(I_\alpha\ast|u|^{p}|v|^q)\\
&>0,\endaligned$$ which is a contradiction. Hence $\mathcal{P}'(u,v)\neq0$ for any $(u,v)\in\mathcal{N}_0$.

{\bf Step 4.} Every critical point of $\mathcal{I}|_{\mathcal{N}_0}$ is a critical point of $\mathcal{I}$ in $H.$

If $(u,v)$ is a critical point of $\mathcal{I}|_{\mathcal{N}_0}$, i.e., $(u,v)\in \mathcal{N}_0$ and $(\mathcal{I}|_{\mathcal{N}_0})'(u,v)=0.$ Thanks to the Lagrange multiplier rule, there exists $\rho\in \mathbb R$ such that $\mathcal{I}'(u,v)=\rho\mathcal{P}'(u,v)$, i.e., $0=\langle\mathcal{I}'(u,v),(u,v)\rangle=\rho \langle\mathcal{P}'(u,v),(u,v)\rangle$, by {\bf Step 1}, we know that $\langle\mathcal{P}'(u,v),(u,v)\rangle\neq0$, we deduce that $\rho=0$, and then $\mathcal{I}'(u,v)=0$.
\end{proof}

The following Lemma is a Br\'{e}zis-Lieb lemma for the nonlocal term of the functional.

\begin{Lemma}\label{le2.2}
Let $u_{n}\rightharpoonup u$ and $v_{n}\rightharpoonup v$ in $H^{1}(\mathbb R^{N})$. If $u_{n}\rightarrow u$ and $v_{n}\rightarrow v$  almost everywhere on $\mathbb R^{N}$ as $n\rightarrow\infty$, then
$$\aligned&\ \lim_{n\rightarrow\infty}\int_{\mathbb R^{N}}(I_{\alpha}\ast|u_{n}|^{p})|v_{n}|^{q}-\int_{\mathbb R^{N}}(I_{\alpha}\ast|u|^{p})|v|^{q}\\
=&\ \lim_{n\rightarrow\infty}\int_{\mathbb R^{N}}(I_{\alpha}\ast|u_{n}-u|^{p})|v_{n}-v|^{q}.\endaligned$$

\end{Lemma}
\begin{proof}
 For every $n\in\mathbb N$, one has
$$\begin{aligned}
&\int_{\mathbb R^{N}}(I_{\alpha}\ast|u_n|^p) |v_n|^q-\int_{\mathbb R^{N}}(I_{\alpha}\ast|u_n-u|^p) |v_n-v|^q\\
=&\int_{\mathbb R^{N}}(I_{\alpha}\ast(|u_n|^p-|u_n-u|^p))|v_n|^q\\
&\ +\int_{\mathbb R^{N}}(I_{\alpha}\ast|u_n-u|^p)(|v_n|^q-|v_n-v|^q).
\end{aligned}$$
Since $u_n\rightharpoonup u$ in $H^1(\mathbb{R}^N)$, by \cite[Lemma 2.5]{mor} with $q= p$ and $r =\frac{2Np}{N+\alpha}$, one has
$\int_{\mathbb R^{N}}(|u_n|^p-|u_n-u|^p-|u|^p)^{\frac{2N}{N+\alpha}}\rightarrow0\ \ \hbox{as}\ \  n\rightarrow\infty,$
which means that $|u_n|^p-|u_n-u|^p\rightarrow|u|^p$ in $L^{\frac{2N}{N+\alpha}}(\mathbb{R}^N)$. By (\ref{eqHLS}),
the Riesz potential is a linear bounded map from $L^{\frac{2N}{N+\alpha}}(\mathbb R^N)$ to $L^{\frac{2N}{N-\alpha}}(\mathbb R^N)$, which implies that $I_{\alpha}\ast(|u_n|^p-|u_n-u|^p)\rightarrow I_{\alpha}\ast|u|^p$  in $L^{\frac{2N}{N-\alpha}}(\mathbb{R}^N)$. Since $v_n\rightharpoonup v$ in $H^1(\mathbb{R}^N)$, by $|v_n|^q\rightharpoonup|v|^q$ in $L^{\frac{2N}{N+\alpha}}(\mathbb{R}^N)$, we may obtain that
$\int_{\mathbb R^{N}}(I_{\alpha}\ast(|u_n|^p-|u_n-u|^p))|v_n|^q\rightarrow\int_{\mathbb R^{N}}(I_{\alpha}\ast|u|^p)|v|^q\ \ \hbox{as}\ \ n\rightarrow\infty.$
Similarly, according to $v_n\rightharpoonup v$ in $H^1(\mathbb{R}^N)$, by \cite[Lemma 2.5]{mor} with $r =\frac{2Np}{N+\alpha}$, one has $|v_n|^q-|v_n-v|^q\rightarrow |v|^q$ in $L^{\frac{2N}{N+\alpha}}(\mathbb{R}^N)$.
Since $|u_n-u|^p\rightharpoonup 0$ in $L^{\frac{2N}{N+\alpha}}(\mathbb{R}^N)$, by the Riesz potential is a linear bounded map from $L^{\frac{2N}{N+\alpha}}(\mathbb R^N)$ to $L^{\frac{2N}{N-\alpha}}(\mathbb R^N)$ and the Hardy-Littlewood-Sobolev inequality \eqref{eqHLS}, this implies that $I_{\alpha}\ast|u_n-u|^p\rightharpoonup 0$ in $L^{\frac{2N}{N-\alpha}}(\mathbb{R}^N)$, we may obtain that
$\int_{\mathbb R^{N}}(I_{\alpha}\ast|u_n-u|^p)(|v_n|^q-|v_n-v|^q)\rightarrow0$ as $n\rightarrow\infty,$
we reach the conclusion.
\end{proof}

\mysection {The proof of Theorem \ref{th1.1}}

In this section, we consider $B(x)\equiv B>0$ and prove Theorem \ref{th1.1}.
\vskip5pt

\noindent{\bf Proof of Theorem \ref{th1.1}.} Let $(u_{n},v_{n})\in\mathcal{N}_0$ be a minimizing sequence for $c_0$, which was given in (\ref{eqc}). We first show that $\{(u_n,v_n)\}$ is bounded in $H$. For $n$ large enough, we get that
\begin{equation}\label{eq3}
\aligned &\ c_0+o_n(1)\\
\geq&\ \mathcal{I}(u_n,v_n)-\frac{1}{p+q}\mathcal{P}(u_n,v_n)\\
=&\ \frac{p+q-2}{2(p+q)}\int_{\mathbb R^N}(|\nabla u_n|^2+|\nabla v_n|^2+A(x) |u_n|^2+B|v_n|^2).\endaligned
\end{equation}
Then, there exist a subsequence of $\{u_{n}\}$, $\{v_{n}\}$ (still denoted by $\{u_{n}\}$, $\{v_{n}\}$) such that $u_{n}\rightharpoonup u$ and $v_{n}\rightharpoonup v$ in $H^{1}(\mathbb R^{N})$. This implies in particular that $\{|u_{n}|^p\}$ and $\{|v_{n}|^q\}$ are bounded in $L^{\frac{2N}{N+\alpha}}(\mathbb R^{N})$, $p,q\in(\frac{N+\alpha}{N},2_*^\alpha)$. The proof consists of three steps.
\vskip4pt
{\bf Step 1.} $\int_{\mathbb R^{N}}(I_\alpha\ast|u_{n}|^{p})|v_{n}|^{q}\not\rightarrow0$.
\vskip4pt
From the {\bf Step 2} in Lemma \ref{le2.3} and
$$\aligned\mathcal{I}(u_{n},v_{n})&=\frac{1}{2}\int_{\mathbb R^{N}}(|\nabla u_n|^2+|\nabla v_n|^2+A(x) |u_n|^2+B|v_n|^2)\\
&\ \ \ \ -\frac{2}{p+q}\int_{\mathbb R^{N}}(I_\alpha\ast|u_{n}|^{p})|v_{n}|^{q}\rightarrow c_{0}>0,\endaligned$$
we can obtain $\|(u_{n},v_{n})\|\not\rightarrow0.$ By using Lemma \ref{le22}, for any $t > 1$,
$$\aligned c_{0}\leftarrow\mathcal{I}(u_{n},v_{n})&\geq\mathcal{I}(tu_{n},tv_{n})\\
&=\frac{t^{2}}{2}\int_{\mathbb R^{N}}(|\nabla u_n|^2+|\nabla v_n|^2+A(x) |u_n|^2+B|v_n|^2)\\
&\ \ \ \ -\frac{2t^{p+q}}{p+q}\int_{\mathbb R^{N}}(I_\alpha\ast|u_{n}|^{p})|v_{n}|^{q}\\
&\geq\frac{t^{2}}{2}\delta-\frac{2t^{p+q}}{p+q}\int_{\mathbb R^{N}}(I_\alpha\ast|u_{n}|^{p})|v_{n}|^{q},\endaligned$$
where $\delta$ is a fixed constant. It suffices to take $t>1$ so that $\frac{t^{2}\delta}{2}>2c_{0}$ to get a lower bound for $\int_{\mathbb R^{N}}(I_\alpha\ast|u_{n}|^{p})|v_{n}|^{q}$. Therefore, we may assume that
\begin{equation}\label{eq0.1}
\int_{\mathbb R^{N}}(I_\alpha\ast|u_{n}|^{p})|v_{n}|^{q}\rightarrow D\in(0,\infty).
\end{equation}

\vskip4pt
{\bf Step 2.} Splitting by concentration-compactness.
\vskip4pt
By using (\ref{eq0.1}), the semigroup property of the Riesz potential $I_\alpha=I_{\frac{\alpha}{2}}\ast I_{\frac{\alpha}{2}}$, Cauchy-Schwarz inequality, \eqref{eqHLS},  and the concentration-compactness lemma of Lions \cite[Lemma 2.1]{wi}, there exist $\delta>0$ and $\{x_{n}\}\subset\mathbb R^{N}$  such that
\begin{equation}\label{eq3.3}
\int_{B_{x_{n}}(1)}|u_{n}|^{\frac{2Np}{N+\alpha}}>\delta>0.
\end{equation}
Let $\eta_{R}(t)$ be a smooth function defined on $[0,+\infty)$ such that
\vskip4pt
a) $\eta_{R}(t)=1$ for $0\leq t\leq R;$

b) $\eta_{R}(t)=0$ for $t\geq2R$;

c) $\eta'_{R}(t)\leq\frac{2}{R}.$
\vskip4pt
Define
$$ \theta_{n}(x)=\eta_{R}(|x-x_{n}|)u_{n}(x),$$
$$\lambda_{n}(x)=(1-\eta_{R}(|x-x_{n}|))u_{n}(x),$$
$$\xi_{n}(x)=\eta_{R}(|x-x_{n}|)v_{n}(x)$$
and
$$\mu_{n}(x)=(1-\eta_{R}(|x-x_{n}|))v_{n}(x).$$
Obviously, $(\theta_{n},\xi_{n}),(\lambda_{n},\mu_{n})\in X$ and
$(u_{n},v_{n})=( \theta_{n},\xi_{n})+(\lambda_{n},\mu_{n})=( \theta_{n}+\lambda_{n},\xi_{n}+\mu_{n}).$ Observe that in particular
\begin{equation}\label{eq03.4}
\liminf_{n\rightarrow+\infty}\int_{B_{x_{n}}(R)}|\theta_{n}|^{\frac{2Np}{N+\alpha}}\geq\delta.
\end{equation}
\vskip4pt

{\bf Step 3.}
There exists $\tau> 0$ independent of $\varepsilon$ and $n_{0}= n_{0}(\varepsilon)$ such that $\|(\lambda_{n},\mu_{n})\|\leq C\varepsilon^{\tau}$ for all $ n\geq n_{0}$.
\vskip4pt
Let $(\omega_{n},\nu_{n})=(u_{n}(\cdot+ x_{n}),v_{n}(\cdot+ x_{n}))$, where $\{x_{n}\}$ is given in (\ref{eq3.3}). Clearly, $\omega_{n}\rightharpoonup \omega $ and $\nu_{n}\rightharpoonup \nu$ in $H^{1}(\mathbb R^{N})$. By taking a larger $R$, we may assume that
$\int_{B_{2R\backslash R}}(I_\alpha\ast|\omega|^{\alpha})|\nu|^{\beta}<\varepsilon ,$
where $ B_{2R\backslash R}$ denotes the anulus centered in 0 with radii $R$ and $2R$. Then,  for $n$ large enough,\\

\begin{equation}\label{eq33}
\aligned &\ \bigg| \int_{\mathbb R^{N}}\Bigl((I_\alpha\ast|u_{n}|^{p})|v_{n}|^{q}-(I_\alpha\ast| \theta_{n}|^{p})|\xi_{n}|^{q}-(I_\alpha\ast|\lambda_{n}|^{p})|\mu_{n}|^{q}\Bigr)\bigg|\\
=&\ \int_{\mathbb R^{N}}\Bigl((I_\alpha\ast|u_{n}|^{p})|v_{n}|^{q}-(I_\alpha\ast| \eta_R(|x-x_n|)u_{n}|^{p})|\eta_R(|x-x_n|)v_{n}|^{q}\\
&\ -(I_\alpha\ast|(1-\eta_R(|x-x_n|))u_{n}|^{p})|(1-\eta_R(|x-x_n|))v_{n}|^{q}\Bigr)\\
=&\ \int_{\mathbb R^{N}}(I_\alpha\ast|u_{n}|^{p})|v_{n}|^{q}-\int_{B_R}(I_\alpha\ast| \eta_R(|x-x_n|)u_{n}|^{p})|\eta_R(|x-x_n|)v_{n}|^{q}\\
&\ -\int_{B_{2R\backslash R}}(I_\alpha\ast| \eta_R(|x-x_n|)u_{n}|^{p})|\eta_R(|x-x_n|)v_{n}|^{q}\\
&\ -\int_{B_{2R}^c}(I_\alpha\ast|(1-\eta_R(|x-x_n|))u_{n}|^{p})|(1-\eta_R(|x-x_n|))v_{n}|^{q}\\
&\ -\int_{B_{2R\backslash R} }(I_\alpha\ast|(1-\eta_R(|x-x_n|))u_{n}|^{p})|(1-\eta_R(|x-x_n|))v_{n}|^{q}\\
=&\ \int_{\mathbb R^{N}}(I_\alpha\ast|u_{n}(x+x_n)|^{p})|v_{n}(x+x_n)|^{q}\\
&\ -\int_{B_R}(I_\alpha\ast| \eta_Ru_{n}(x+x_n)|^{p})|\eta_Rv_{n}(x+x_n)|^{q}\\
&\ -\int_{B_{2R\backslash R}}(I_\alpha\ast| \eta_Ru_{n}(x+x_n)|^{p})|\eta_Rv_{n}(x+x_n)|^{q}\\
&\ -\int_{B_{2R}^c}(I_\alpha\ast|(1-\eta_R)u_{n}(x+x_n)|^{p})|(1-\eta_R)v_{n}(x+x_n)|^{q}\\
&\ -\int_{B_{2R\backslash R} }(I_\alpha\ast|(1-\eta_R)u_{n}(x+x_n)|^{p})|(1-\eta_R)v_{n}(x+x_n)|^{q}\\
=&\ \int_{\mathbb R^{N}}(I_\alpha\ast|\omega_{n}|^{p})|\nu_{n}|^{q}-\int_{B_R}(I_\alpha\ast| \omega_n|^{p})|\nu_n|^{q}\\
&\ -\int_{B_{2R\backslash R}}(I_\alpha\ast| \eta_R\omega_{n}|^{p})|\eta_R\nu_{n}|^{q}-\int_{B_{2R}^c}(I_\alpha\ast|\omega_{n}|^{p})|\nu_{n}|^{q}\\
&\ -\int_{B_{2R\backslash R} }(I_\alpha\ast|(1-\eta_R)\omega_{n}|^{p})|(1-\eta_R)\nu_{n}|^{q}\\
=&\ \int_{B_{2R\backslash R}}(1-\eta_R^{p+q}-(1-\eta_R)^{p+q})(I_\alpha\ast|\omega_n|^p)|\nu_n|^q\\
\leq&\ 3\varepsilon.\endaligned
\end{equation}
Since $|\nabla \omega_{n}|^{2}$ and $|\nabla \nu_{n}|^{2}$ are uniformly bounded in  $L^{1}(\mathbb R^{N})$, up to a subsequence,  $|\nabla \omega_{n}|^{2}$,  $|\nabla \nu_{n}|^{2}$ converge (in the sense of measure) to a certain positive measure  $\mu(\mathbb R^{N})<+\infty$.  By enlarging $R$, we may assume that $\mu(B_{2R\backslash R})<\varepsilon$.  Thus, for $R$ large enough,
$$\int_{B_{2R\backslash R}}|\nabla \omega_{n}|^{2} =\int_{B_{2R\backslash R}}d\mu<\varepsilon ,\ \ \ \ \int_{B_{2R\backslash R}}|\nabla \nu_{n}|^{2} =\int_{B_{2R\backslash R}}d\mu<\varepsilon ,$$
by using  H\"{o}lder inequality, we have
\begin{equation}\label{eq34}
\aligned \bigg| \int_{\mathbb R^{N}}(|\nabla u_{n}|^{2} -|\nabla  \theta_{n}|^{2}-|\nabla \lambda_{n}|^{2}) \bigg|=\bigg|2\int_{\mathbb R^{N}}\nabla \theta_{n}\nabla \lambda_{n}\bigg|\leq C\varepsilon,\endaligned
\end{equation}
\begin{equation}\label{eq35}
\aligned \bigg| \int_{\mathbb R^{N}}(|\nabla v_{n}|^{2} -|\nabla \xi_{n}|^{2}-|\nabla \mu_{n}|^{2}) \bigg|=\bigg|2\int_{\mathbb R^{N}}\nabla \xi_{n}\nabla \mu_{n}\bigg|\leq C\varepsilon.\endaligned
\end{equation}
Since $\omega_n\rightarrow \omega,\nu_n\rightarrow\nu$ in $L_{loc}^2(\mathbb R^{N})$, take $R$ such that $\int_{B_{2R\backslash R}}|\omega|^2<\varepsilon$ and $\int_{B_{2R\backslash R}}|\nu|^2<\varepsilon$, then
\begin{equation}\label{eq37}
\bigg|\int_{\mathbb R^{N}}\bigl(A(x)u_{n}^{2} -A(x)\theta_{n}^{2}-A(x)\lambda_{n}^{2}\bigr) \bigg|<C\varepsilon,\ \
\end{equation}
\begin{equation}\label{eq38}
 \bigg|\int_{\mathbb R^{N}}B(v_{n}^{2} -\xi_{n}^{2}-\mu_{n}^{2}) \bigg|< C\varepsilon.\ \ \ \ \ \ \ \ \ \
\end{equation}
By (\ref{eq33})-(\ref{eq38}) we obtain that for $n$ sufficiently large and $t>0$,
\begin{equation}\label{eq39}
\aligned &\ \bigg| \mathcal{I}( tu_{n},tv_{n})-\mathcal{I}(t\theta_{n},t\xi_{n})-\mathcal{I}(t\lambda_{n},t\mu_{n})\bigg|\\
\leq&\ C\varepsilon(t^{2}+t^{p+q}).\endaligned
\end{equation}
Let us denote with $t_{1},t_{2}>0$  which maximize $h_{\theta_{n}\xi_{n}}(t)$ and $h_{\lambda_{n}\mu_{n}}(t)$ respectively, i.e.,
$$\mathcal{I}(t_1\theta_{n},t_1\xi_{n})=\max_{t>0}\mathcal{I}(t\theta_{n},t\xi_{n}),\ $$ $$\mathcal{I}(t_2\lambda_{n},t_2\mu_{n})=\max_{t>0}\mathcal{I}(t\lambda_{n},t\mu_{n}).$$
First, let us assume that $t_{1}\leq t_{2}$. Then,
\begin{equation}\label{eq310}
\mathcal{I}(t\lambda_{n},t\mu_{n})\geq0,\ \   t\leq t_{1}.
\end{equation}
\vskip4pt
{\bf Claim:} There exist $0<\tilde{t}<1<\hat{t}$ independent of  $\varepsilon$ such that $t_{1}\in(\tilde{t},\hat{t})$.

Indeed, take $\hat{t}=\left((p+q)(D+1)^{-1}L\right)^{\frac{1}{p+q-2}}$, where $D$ comes from (\ref{eq0.1}) and $L$ is large enough such that  $\hat{t}>1$ and
\begin{equation}\label{eq311}
L\geq \int_{\mathbb R^{N}}(|\nabla u_{n}|^{2}+|\nabla v_{n}|^{2}+A_{\infty}u_{n}^{2}+ B v_{n}^{2}),
\end{equation}
then
$$\aligned \mathcal{I}(\hat{t}u_{n},\hat{t}v_{n})&= \frac{\hat{t}^{2}}{2}\int_{\mathbb R^{N}}\Bigl(|\nabla u_{n}|^{2}+|\nabla v_{n}|^{2}+A(x)u_{n}^{2}+ Bv_{n}^{2} \\
&\ \ \ \ -\frac{2\hat{t}^{p+q}}{p+q} (I_\alpha\ast|u_{n}|^{p})|v_{n}|^{q}\Bigr)\\
&\leq\frac{\hat{t}^{2}}{2}\int_{\mathbb R^{N}} \Bigl( |\nabla u_{n}|^{2}+|\nabla v_{n}|^{2}+A_\infty u_{n}^{2}+ Bv_{n}^{2}\\
&\ \ \ \ -\frac{4\hat{t}^{p+q-2}}{p+q}(I_\alpha\ast|u_{n}|^{p})|v_{n}|^{q}\Bigr) \\
&\leq-\frac{3}{2}L\hat{t}^{2}\\
&<0.\endaligned $$
It follows from (\ref{eq39}) that
\begin{equation}\label{eq3111}
\aligned &\ \mathcal{I}(tu_{n},tv_{n}) \geq \mathcal{I}(t\theta_{n},t\xi_{n})+\mathcal{I}(t\lambda_{n},t\mu_{n})-C\varepsilon,\ \ t\in(0,\hat{t}].\endaligned
\end{equation}
Then, taking a smaller $\varepsilon$,
$$\mathcal{I}(\hat{t}\theta_{n},\hat{t}\xi_{n})+\mathcal{I}(\hat{t}\lambda_{n},\hat{t}\mu_{n})<0.$$
Then $\mathcal{I}(\hat{t}\theta_{n},\hat{t}\xi_{n})<0$ or  $\mathcal{I}(\hat{t}\lambda_{n},\hat{t}\mu_{n})<0$. In any case Lemma \ref{le22} implies that $t_{1}<\hat{t}$.

Take  $\tilde{t}=\left(\frac{c_0}{L}\right)^{\frac{1}{2}}$, where $L$ is chosen as in (\ref{eq311}).  Then $\tilde{t}<1$. For any $t\leq \tilde{t}$,
$$\aligned \mathcal{I}(tu_{n},tv_{n})&\leq \frac{\tilde{t}^{2}}{2}\int_{\mathbb R^{N}}(|\nabla u_{n}|^{2}+|\nabla v_{n}|^{2}+A(x)u_{n}^{2}+ B v_{n}^{2})\\
&\leq \frac{c_0}{2L}\int_{\mathbb R^{N}}(|\nabla u_{n}|^{2}+|\nabla v_{n}|^{2}+A_{\infty}u_{n}^{2}+ B v_{n}^{2})\\
&\leq\frac{c_0}{2}.\endaligned $$
By (\ref{eq310}), (\ref{eq3111}) and $\mathcal{I}(t_1\theta_{n},t_1\xi_{n})\geq c_0$, we obtain
\begin{equation}\label{eq313}
\aligned \mathcal{I}(t_1u_{n},t_1v_{n})&\geq \mathcal{I}(t_1\theta_{n},t_1\xi_{n})+\mathcal{I}(t_1\lambda_{n},t_1\mu_{n})-C\varepsilon\\
&\geq c_0-C\varepsilon,\endaligned
\end{equation}
by choosing a small $\varepsilon$, we have that $\mathcal{I}(t_1u_{n},t_1v_{n})\geq \frac{c_0}{2}$. Therefore, $t_{1}>\tilde{t}$. Since $(u_{n},v_{n})\in \mathcal{N}_0$, $h_{uv}$ reaches its maximum at  $t=1$. Then,
$$c_0\leftarrow \mathcal{I}(u_{n},v_{n})\geq \mathcal{I}(t_1u_{n},t_1v_{n}), $$
and using (\ref{eq313}) we deduce, for $n$ large and $t\in(0,t_{1})$,
$$\aligned \mathcal{I}(t\lambda_{n},t\mu_{n})&\leq \mathcal{I}(t_1\lambda_{n},t_1\mu_{n})\\
&\leq \mathcal{I}(t_1u_{n},t_1v_{n})-\mathcal{I}(t_1\theta_{n},t_1\xi_{n})+C\varepsilon\\
&\leq \mathcal{I}(u_{n},v_{n})-c_0+C\varepsilon\\
&\leq2C \varepsilon.\endaligned$$ Moreover, for any $t\in(0,\tilde{t}),$
$$\aligned 2C\varepsilon&\geq \mathcal{I}(t\lambda_{n},t\mu_{n})\\
&=\frac{t^{2}}{2}\int_{\mathbb R^{N}}(|\nabla \lambda_{n}|^{2}+|\nabla \mu_{n}|^{2}+A(x)\lambda_{n}^{2}+ B\mu_{n}^{2})\\
&\ \ \ \ -\frac{2}{p+q}t^{p+q}\int_{\mathbb R^{N}}(I_\alpha\ast|\lambda_{n}|^{p})|\mu_{n}|^{q}\\
&\geq \frac{t^{2}}{2}q_{n}-Kt^{p+q},\endaligned$$
where
$$q_{n}=\int_{\mathbb R^{N}}(|\nabla \lambda_{n}|^{2}+|\nabla \mu_{n}|^{2}+A_{0}\lambda_{n}^{2}+B\mu_{n}^{2})$$
is bounded and $K>D$. Observe that $ \frac{t^{2}}{2}q_{n}-Kt^{p+q}=\frac{t^{2}}{4}q_{n}$ for $t=\left(\frac{q_{n}}{4K}\right)^{\frac{1}{p+q-2}}$. By taking a larger $K$ we may assume that $\left(\frac{q_{n}}{4K}\right)^{\frac{1}{p+q-2}}\leq \tilde{t}$. Then, we obtain
$$2C\varepsilon\geq \mathcal{I}(t\lambda_{n},t\mu_{n})\geq \left(\frac{q_{n}}{4K}\right)^{\frac{2}{p+q-2}}\frac{q_{n}}{4 }\geq Cq_{n}^{\frac{p+q}{p+q-2}},$$
which implies that
\begin{equation}\label{eq314}
\|(\lambda_{n},\mu_{n})\|\leq C\varepsilon^{\frac{p+q-2}{2(p+q)}}.
\end{equation}

In the case  $t_{1}>t_{2}$, we can argue similarly to conclude that $\|( \theta_{n},\xi_{n})\|\leq C\varepsilon^{\frac{p+q-2}{2(p+q)}}$. But, taking small  $\varepsilon$, this contradicts (\ref{eq3.3}), so (\ref{eq314}) holds.

\vskip4pt
{\bf Step 4:} The infimum of  $\mathcal{I}|_{\mathcal{N}_0}$ is achieved.
\vskip4pt
By $(\omega_{n},\nu_{n})=(u_{n}(\cdot+x_{n}),v_{n}(\cdot+x_{n})),$ $\omega_{n}\rightharpoonup \omega$ and $\nu_{n}\rightharpoonup \nu$ in $H^{1}(\mathbb R^{N})$. Moreover, by  Rellich-Kondrachov theorem, we know that $\omega_{n}\rightarrow \omega$ in $L_{loc}^{2}(\mathbb R^{N}),\nu_{n}\rightarrow\nu$ in $L_{loc}^{2}(\mathbb R^{N}).$ Then $\omega\neq0$ since, by (\ref{eq3.3}),
$$ \delta<\liminf_{n\rightarrow+\infty}\int_{\mathbb R^{N}}| \theta_{n}|^{\frac{2Np}{N+\alpha}}\leq\lim_{n\rightarrow+\infty}\int_{B_{0}(2R)}|\omega_{n}|^{\frac{2Np}{N+\alpha}}=\int_{B_{0}(2R)}|\omega|^{\frac{2Np}{N+\alpha}}.$$
Since $(u_{n},v_{n})=(\theta_{n},\xi_{n})+(\lambda_{n},\mu_{n})$, moreover
$$\|(\lambda_{n},\mu_{n})\|\leq C\varepsilon^{\frac{p+q-2}{2(p+q)}}.$$
By using H\"{o}lder inequality, one has
\begin{equation}\label{eq315}
\aligned \int_{\mathbb R^{N}}|u_{n}^{2}-\theta_{n}^{2}|&\leq \int_{\mathbb R^{N}}|\lambda_{n}|(|u_{n}|+| \theta_{n}|)\\
&\leq \bigg(\int_{\mathbb R^{N}}\lambda_{n}^{2}\bigg)^{\frac{1}{2}}\bigg(\int_{\mathbb R^{N}}(|u_{n}|+| \theta_{n}|)^{2}\bigg)^{\frac{1}{2}}\\
&\leq C\varepsilon^{\frac{p+q-2}{2(p+q)}},\endaligned
\end{equation}
and
\begin{equation}\label{eq316}
\aligned \int_{\mathbb R^{N}}|v_{n}^{2}-\xi_{n}^{2}|&\leq \int_{\mathbb R^{N}}|\mu_{n}|(|v_{n}|+|\xi_{n}|)\\
&\leq \bigg(\int_{\mathbb R^{N}}\mu_{n}^{2}\bigg)^{\frac{1}{2}}\bigg(\int_{\mathbb R^{N}}(|v_{n}|+|\xi_{n}|)^{2}\bigg)^{\frac{1}{2}}\\
&\leq C\varepsilon^{\frac{p+q-2}{2(p+q)}}.
\endaligned
\end{equation}
Furthermore,
$$\int_{\mathbb R^{N}}\theta_{n}^{2}\leq \int_{B_{0}(2R)}\omega_{n}^{2}\rightarrow\int_{B_{0}(2R)}\omega^{2}\leq\int_{\mathbb R^{N}}\omega^{2},  $$
$$\int_{\mathbb R^{N}}\xi_{n}^{2}\leq \int_{B_{0}(2R)}\omega_{n}^{2}\rightarrow\int_{B_{0}(2R)}\nu^{2}\leq\int_{\mathbb R^{N}}\nu^{2}.  $$
Combining these estimate with (\ref{eq315}), (\ref{eq316}), we get
$$\liminf_{n\rightarrow+\infty}\int_{\mathbb R^{N}}\omega_{n}^{2}=\liminf_{n\rightarrow+\infty}\int_{\mathbb R^{N}}u_{n}^{2}\leq \int_{\mathbb R^{N}}\omega^{2}+C\varepsilon^{\frac{p+q-2}{2(p+q)}},$$
$$\liminf_{n\rightarrow+\infty}\int_{\mathbb R^{N}}\nu_{n}^{2}=\liminf_{n\rightarrow+\infty}\int_{\mathbb R^{N}}v_{n}^{2}\leq \int_{\mathbb R^{N}}\nu^{2}+C\varepsilon^{\frac{p+q-2}{2(p+q)}}.\ $$
Since $\varepsilon$ is arbitrary, we obtain that $\omega_{n}\rightarrow \omega, \ \nu_{n}\rightarrow  \nu$ in $ L^{2}(\mathbb R^{N})$, by using interpolation inequality,  $ \omega_{n}\rightarrow \omega\neq0 $ in $L^{\frac{2Np}{N+\alpha}}(\mathbb R^{N})$, $p \in(\frac{N+\alpha}{N},2_*^\alpha)$.
\vskip4pt
We discuss two cases:
\vskip4pt
{\bf Case 1:} $\{x_{n}\}$ is bounded. Assume, passing to a subsequence, that $x_{n}\rightarrow x_{0}$. In this case $u_{n}\rightharpoonup u,v_{n}\rightharpoonup v$ in $H^{1}(\mathbb R^{N})$, $u_{n}\rightarrow u $  in $L^{\frac{2Np}{N+\alpha}}(\mathbb R^{N})$, $ v_{n}\rightarrow v$ in $L^{\frac{2Nq}{N+\alpha}}(\mathbb R^{N})$ for $p,\ q\in(\frac{N+\alpha}{N},2_*^\alpha)$, where $u=\omega(\cdot-x_{0}),v= \nu(\cdot-x_{0}) $.  By using Lemma \ref{le22}, \eqref{009}, Lemma \ref{le2.2} and the weak lower semi-continuity of the norm, we obtain
$$c_0=\lim_{n\rightarrow+\infty}\mathcal{I}(u_{n} ,v_{n})\geq\liminf_{n\rightarrow+\infty} \mathcal{I}(tu_{n},tv_{n})\geq \mathcal{I}(tu,tv),\ \ t>0,$$
therefore, $\max\limits_{t>0}\mathcal{I}(tu,tv)=c_0$, $u_{n}\rightarrow u,v_{n}\rightarrow v$ in $H^{1}(\mathbb R^{N})$. In particular, $(u,v)\in\mathcal{N}_0$ is a minimizer of $\mathcal{I}|_{\mathcal{N}_0}$.
\vskip4pt
{\bf Case 2:} $\{x_{n}\}$ is unbounded. In this case, by $(F_{1})$ and Lebesgue dominated convergence theorem,
$$\aligned \lim_{n\rightarrow+\infty}\int_{\mathbb R^{N}}A(x)u_{n}^{2}&=\lim_{n\rightarrow+\infty}\int_{\mathbb R^{N}}A(x+x_{n})\omega_{n}^{2}\\
&=A_{\infty}\int_{\mathbb R^{N}}\omega^{2}\\
&\geq\int_{\mathbb R^{N}}A(x)\omega^{2}\\
&=\lim_{n\rightarrow+\infty}\int_{\mathbb R^{N}}A(x)\omega_{n}^{2}.\endaligned $$
 Then, by using Lemma \ref{le22}, \eqref{009}, Lemma \ref{le2.2} and the weak lower semi-continuity of the norm, for every $t>0$,
$$\aligned c_0&=\lim_{n\rightarrow+\infty}\mathcal{I}(u_{n},v_{n})\\
&\geq \liminf_{n\rightarrow+\infty} \mathcal{I}(tu_{n},tv_{n})\\
&\geq\liminf_{n\rightarrow+\infty} \mathcal{I}(t\omega_{n},t\nu_{n})\\
&\geq \mathcal{I}(t\omega,t\nu).\endaligned$$
Taking $t^{z}$ so that $h_{\omega v}(t)=\mathcal{I}(t\omega,t\nu)$ reaches its maximum, we conclude that $(\omega_{t^{z}},\nu_{t^{z}})\in\mathcal{N}_0$ and is a minimizer for $\mathcal{I}|_{\mathcal{N}_0}$. By Lemma \ref{le2.3}, the $(u,v)$ is a critical point of $\mathcal{I}$ and hence a solution to \eqref{eq1.1}. We claim $u,\ v\neq0$. We already know that $u\neq0$.  Now we prove $v\neq0$. Indeed, if $v=0$, then the first equation of (\ref{eq1.1}) yields that $u=0$, then $(u,v)=(0,0)$, this is impossible by the {\bf Step 1} in Lemma \ref{le2.3}. In addition, $(|u|,|v|)\in H$ is also a solution. By applying the maximum principle to each equation of \eqref{eq1.1}, we get that $u,v>0$. The proof is complete.

\mysection {The proof of Theorem \ref{th1.2}}

In this section, we consider the case of periodic potential.

\noindent{\bf Proof of Theorem \ref{th1.2}.} The proof is similar to Theorem \ref{th1.1}. Here we only point out the difference.  We claim that $(u,v)\in\mathcal{N}_0.$ Indeed, if $(u,v)\not\in\mathcal{N}_0,$ we discuss it in two cases: $\mathcal{P}(u,v)<0$ and $\mathcal{P}(u,v)>0$.

{\bf Case 1:} $\mathcal{P}(u,v)<0$. By Lemma \ref{le22}, there exists $t\in(0,1)$ such that $(tu,tv)\in \mathcal{N}_0.$ From $(u_{n},v_{n})\in\mathcal{N}_0$ and Fatou's Lemma that
$$\aligned c_0&=\liminf_{n\rightarrow+\infty}\bigg(\mathcal{I}(u_n,v_n)-\frac{1}{p+q}\mathcal{P}(u_n,v_n)\bigg)\\
&=\frac{p+q-2}{2(p+q)}\liminf_{n\rightarrow+\infty}\int_{\mathbb R^{N}}(|\nabla u_n|^2+|\nabla v_n|^2+A(x) |u_n|^2+B(x)|v_n|^2)\\
&\geq\frac{p+q-2}{2(p+q)}\int_{\mathbb R^{N}}(|\nabla u|^{2}+|\nabla v|^{2}+A(x)|u|^2+B(x)|v|^{2})\\
&>\frac{p+q-2}{2(p+q)}t^{2}\int_{\mathbb R^{N}}(|\nabla u|^{2}+|\nabla v|^{2}+A(x)|u|^2+B(x)|v|^{2})\\
&=\mathcal{I}(tu,tv)-\frac{1}{p+q}\mathcal{P}(tu,tv)\\
&\geq c_0,\endaligned$$
which is a contradiction.
\vskip4pt
{\bf Case 2:} $\mathcal{P}(u,v)>0$. We define $\xi_{n}:=u_{n}-u,\gamma_{n}:=v_{n}-v$. By using Br\'{e}zis-Lieb lemma \cite[Lemma 1.32]{wi} and Lemma \ref{le2.2}, we may obtain
\begin{equation}\label{eqI}
\mathcal{I}(u_{n},v_{n})=\mathcal{I}(u,v)+\mathcal{I}(\xi_{n},\gamma_{n})+o_{n}(1),
\end{equation}
and
\begin{equation}\label{eqI1}
\mathcal{P}(u_{n},v_{n})=\mathcal{P}(u,v)+\mathcal{P}(\xi_{n},\gamma_{n})+o_{n}(1).
\end{equation}
Threfore $\limsup\limits_{n\rightarrow\infty}\mathcal{P}(\xi_{n},\gamma_{n})<0$. According to Lemma \ref{le22}, there is a  $t_{n}\in(0,1)$ such that $(t_{n}\xi_{n},t_{n}\gamma_{n})\in \mathcal{N}_0.$ Furthermore, one has that $\limsup\limits_{n\rightarrow\infty}t_{n}<1$, otherwise, along a subsequence, $t_{n}\rightarrow1$ and hence $\mathcal{P}(\xi_{n},\gamma_{n})=\mathcal{P}(t_{n}\xi_{n}, t_{n}\gamma_{n})+o_{n}(1)=o_{n}(1),$
which is a contradiction. For $n$ large enough, it follows from $(u_{n},v_{n})\in \mathcal{N}_0$, (\ref{eqI}) and (\ref{eqI1}) that
$$\aligned  &\ c_0+o_{n}(1)\\
=&\ \mathcal{I}(u_{n},v_{n})-\frac{1}{p+q}\mathcal{P}(u_{n},v_{n})\\
=&\ \frac{p+q-2}{2(p+q)}\int_{\mathbb R^{N}}(|\nabla u_n|^{2}+|\nabla v_n|^{2} +A(x)|u_n|^{2}+B(x)|v_n|^{2} )\\
=&\ \frac{p+q-2}{2(p+q)}\int_{\mathbb R^{N}}(|\nabla u|^{2}+|\nabla v|^{2}+|\nabla \xi_{n}|^{2}+|\nabla \gamma_{n}|^{2}\\
&\ +A(x)|v|^{2}+A(x)|\xi_{n}|^{2}+B(x)|v|^{2}+B(x)|\gamma_{n}|^{2})\\
>&\ \frac{p+q-2}{2(p+q)}\int_{\mathbb R^{N}}(|\nabla u|^{2}+|\nabla v|^{2}+A(x)|u |^{2}+B(x)|v|^{2})\\
&\  +\frac{p+q-2}{2(p+q)}t_n^2\int_{\mathbb R^{N}}(|\nabla\xi_{n}|^{2}+|\nabla\gamma_{n}|^{2} +A(x)|\xi_n|^{2} +B(x)|\gamma_{n}|^{2})\\
=&\ \mathcal{I}(u ,v)-\frac{1}{p+q}\mathcal{P}(u,v)+\mathcal{I}(t_{n}\xi_{n},t_{n}\gamma_{n})-\frac{1}{p+q}\mathcal{P}(t_n\xi_{n},t_n\gamma_{n})\\
=&\ \mathcal{I}(t_n\xi_{n},t_n\gamma_{n})+\frac{p+q-2}{2(p+q)}\int_{\mathbb R^{N}}(|\nabla u|^{2}+|\nabla v|^{2}\\
&\ +A(x)|u|^{2}+B(x)|v|^{2})\\
\geq&\ c_0
,\endaligned$$
which is also a contradiction.

Therefore, $(u,v)\in\mathcal{N}_0$ and then $(u,v)$ is a minimizer of $\mathcal{I}|_{\mathcal{N}_0}$ and the rest proof is similar to Theorem \ref{th1.1}.
\vskip10pt
\noindent{\bf Acknowledgments}
\vskip16pt
The authors would like to gratefully acknowledge support by National Natural Science Foundation of China 
(No. 11871152)  and Key Project of Natural Science Foundation of Fujian (No. 2020J02035).
We thank wish to thank the anonymous 
referee vrery much for the careful reading and valuable comments.

\end{document}